\documentclass{article}

\usepackage{amsmath}
\usepackage{fullpage}
\usepackage[colorlinks,linkcolor=blue,citecolor=blue]{hyperref}

\usepackage{enumerate}
\usepackage{amsfonts}

\newcommand{\betadown}{\beta_{\downarrow}}
\newcommand{\wdown}{w_{\downarrow}}
\newcommand{\Gdown}{G_{\downarrow}}
\newcommand{\N}{\mathcal N}
\newcommand{\NS}{\mathcal N_S}

\usepackage{amsthm}
\newtheorem{thm}{Theorem}
\newtheorem{lem}[thm]{Lemma}
\newtheorem{cor}[thm]{Corollary}
\theoremstyle{definition}
\newtheorem{rmk}[thm]{Remark}

\title{A $\{-1,0,1\}$- and sparsest basis for the null space of a forest in optimal time}

\author{Daniel A.\ Jaume\footnote{Departamento de Matem\'atica, Universidad Nacional del San Luis, San Luis, Argentina. E-mail addresses: \texttt{djaume@unsl.edu.ar} (D.A.\ Jaume), \texttt{lgmolina@unsl.edu.ar} (G.\ Molina), and \texttt{agpastine@unsl.edu.ar} (A.\ Pastine)} \and Gonzalo Molina$^*$ \and Adri\'an Pastine$^*$ \and Mart\'in D.\ Safe\footnote{Departamento de Matem\'atica, Universidad Nacional del Sur, Bah\'ia Blanca, Argentina. E-mail address: \texttt{msafe@uns.edu.ar}~(M.D.\ Safe)}}

\begin{document}

\maketitle

\begin{abstract} Given a matrix, the \textsc{Null Space Problem} asks for a basis of its null space having the fewest nonzeros. This problem is known to be NP-complete and even hard to approximate. The null space of a forest is the null space of its adjacency matrix. Sander and Sander~(2005) and Akbari et al.~(2006), independently, proved that the null space of each forest admits a $\{-1,0,1\}$-basis. We devise an algorithm for determining a sparsest basis of the null space of any given forest which, in addition, is a $\{-1,0,1\}$-basis. Our algorithm is time-optimal in the sense that it takes time at most proportional to the number of nonzeros in any sparsest basis of the null space of the input forest. Moreover, we show that, given a forest $F$ on $n$ vertices, the set of those vertices $x$ for which there is a vector in the null space of $F$ that is nonzero at $x$ and the number of nonzeros in any sparsest basis of the null space of $F$ can be found in $O(n)$ time.\end{abstract}

\section{Introduction}

Given a matrix, the \textsc{Null Space Problem}~\cite{colemanPhD} asks for a basis of its null space which is \emph{sparsest} (i.e., has the fewest nonzeros). This problem is known to be NP-complete~\cite{MR857589} and even hard to approximate~\cite{MR3537025}. Some heuristics for solving this problem were proposed in~\cite{MR812616,MR918058,MR897742,MR3537025}. The \emph{null space} of a forest $F$, denoted $\N(F)$, is the null space of its adjacency matrix. A \emph{null basis} of $F$ is a basis of $\N(F)$. Sander and Sander~\cite{MR2154182} and Akbari et al.~\cite{MR2214403}, independently, proved that the null space of each forest admits a $\{-1,0,1\}$-basis (i.e., a basis whose entries are $-1$, $0$, and $1$ only). Moreover, algorithms for finding one such basis for any given forest were also devised in \cite{MR2214403,1709.03865,MR2154182}, but the basis produced by these algorithms are not necessarily sparsest.

Our main result is a combinatorial algorithm for producing a sparsest basis of the null space of any given forest which, in addition, is a $\{-1,0,1\}$-basis. Moreover, our algorithm is time-optimal in the sense that it takes time at most proportional to the number of nonzeros in any sparsest basis of the null space of the input forest.

This work is organized as follows. In Section~\ref{sec:prel}, we give some basic definitions and preliminaries. In Section~\ref{sec:-1,0,1}, we give an algorithm for producing a $\{-1,0,1\}$-basis of a forest in time at most proportional to the number of nonzeros in the output basis and another one that, given a forest $F$ on $n$ vertices, finds in $O(n)$ time the set of vertices $x$ of $F$ for which there is some vector in the null space of $F$ that is nonzero at $x$. In Section~\ref{sec:sparsest}, we give our time-optimal algorithm for producing a $\{-1,0,1\}$- and sparsest null basis of any given forest and show that, given a forest $F$ on $n$ vertices, the number of nonzeros in any sparsest basis of $F$ can be found in $O(n)$ time.

\section{Preliminaries}\label{sec:prel}

All graphs in this work are finite, undirected and with neither loops nor multiple edges. For all graph-theoretic notions not defined here, the reader is referred to~\cite{MR1367739}. For each set $X$, $\vert X\vert$ denotes its cardinality.

Let $G$ be a graph. We denote by $V(G)$ and $E(G)$ its vertex and edge set, respectively. If $x\in V(G)$, we denote by $N_G(x)$ the set of vertices adjacent to $x$ in $G$. If $f:X\to\mathbb R$ is any function and $Y\subseteq X$, we denote by $f(Y)$ the value $\sum_{y\in Y}f(y)$. If $G$ is a graph, we regard each vector $z\in\N(G)$ as a function $z:V(G)\to\mathbb R$ such that $z(N_G(x))=0$ for each $x\in V(G)$.

A \emph{stable set} of $G$ is a set of pairwise nonadjacent vertices of $G$. The \emph{components} (or \emph{connected components}) of $G$ are the maximal connected subgraphs of $G$. If $X\subseteq G$, the \emph{subgraph of $G$ induced by $X$} is the graph that arises from $G$ by removing all the vertices not in $X$. The \emph{length} of a path is its number of edges. A vertex $u$ is \emph{reachable} from a vertex $v$ by a path $P$ if $u$ and $v$ are the endpoints of $P$.

A \emph{matching} of $G$ is a set of pairwise vertex-disjoint edges of $G$. Let $M$ be a matching of $G$. An \emph{$M$-alternating path} of $G$ is a path of $G$ that alternates between edges in $M$ and edges not in $M$. A vertex of $G$ is \emph{$M$-saturated} if it is an endpoint of some edge of $M$, and \emph{$M$-unsaturated} otherwise. An $M$-\emph{augmenting path} is an $M$-alternating path whose both endpoints are $M$-unsaturated. Maximum matchings are characterized as follows.

\begin{thm}[\cite{MR0094811}]\label{thm:Berge} A matching $M$ of a graph $G$ is maximum if and only if $G$ has no $M$-augmenting paths.
\end{thm}

The dimension of the null space of a forest is characterized as follows.

\begin{thm}[\cite{MR0172271}]\label{thm:nullity} If $F$ is a forest and $M$ is a maximum matching of $F$, then $\dim\N(F)=\vert V(F)\vert-2\vert M\vert$.\end{thm}

Coleman and Pothen~\cite{MR857589} proved that a sparsest basis of the null space of any matrix can be built greedily.

\begin{thm}[\cite{MR857589}]\label{thm:greedy} Let $B=\{b_1,\ldots,b_d\}$ a basis of the null space of some matrix $A$. If, for each $i\in\{1,\ldots,d\}$, $b_i$ is sparsest among the vectors in the null space of $A$ that do not belong to the subspace generated by $\{b_1,\ldots,b_{i-1}\}$, then $B$ is a sparsest basis of the null space of $A$.\end{thm}

A \emph{rooted tree} (sometimes also an \emph{in-tree}~\cite{MR1367739}) is a tree $T$ with each edge oriented as \emph{leaving} one of its endpoints and \emph{entering} the other one, in such a way that: for each vertex $v$ there is precisely one edge entering $v$, except precisely for one vertex $r$, called the \emph{root}, having no edge entering it. If so, the rooted tree is also called a \emph{tree rooted at $r$}. If an edge of a rooted tree leaves $u$ and enters $v$, then $v$ is a \emph{child} of $u$ and $u$ is the \emph{parent} of $v$. Each vertex of a rooted tree having no edge leaving it (or, equivalently, having no children) is called a \emph{leaf}.

\section{Finding a $\{-1,0,1\}$-null basis}\label{sec:-1,0,1}

The time bound results in this section are stated assuming that a forest $F$ is given together with some maximum matching $M$. Nevertheless, if a forest $F$ on $n$ vertices is given without a corresponding maximum matching, then a maximum matching of $F$ can be found in $O(n)$ time~\cite{MR0400779}, which keeps the total running time bounds in the results below still valid.

The theorem below gives a procedure for obtaining a $\{-1,0,1\}$-null basis of any forest given along with some maximum matching. The output basis is strongly dependent on the choice of the maximum matching. Our main result will follow by applying the result below to carefully chosen forest and maximum matching.

\begin{thm}\label{thm:B} Let $F$ be a forest and $M$ be a maximum matching of $F$. If $U$ is the set of $M$-unsaturated vertices of $F$ and, for each $u\in U$:
\begin{enumerate}[(i)]
 \item $Q^+(u)$ (resp.\ $Q^-(u)$) is the set of vertices reachable from $u$ by $M$-alternating paths in $F$ of length $4k$ (resp.\ $4k+2$) for any nonnegative integers $k$, and
 
 \item $b_u:V(F)\to\{-1,0,1\}$ such that $b_u(x)$ is $1$ if $x\in Q^+(u)$, $-1$ if $x\in Q^-(u)$, and $0$ otherwise,
\end{enumerate}
then $B=\{b_u\colon\,u\in U\}$ is a $\{-1,0,1\}$-null basis of $F$ and, for each $u\in U$, $b_u$ is the only vector in $B$ that is nonzero at $u$. Moreover, given $F$ and $M$, $B$ can be found in $O(N)$ time, where $N$ is the total number of nonzeros of the vectors in $B$ (i.e., $N=\sum_{u\in U}(\vert Q^+(u)\vert+\vert Q^-(u)\vert)$).\end{thm}
\begin{proof} Let $u$ be an arbitrary vertex of $U$. Notice that $u\in Q^+(u)$ because $u$ is reachable from itself by a zero-length path which is trivially $M$-alternating. Thus, $b_u$ is nonzero at $u$. We will show that $b_u\in\N(F)$; i.e., $b_u(N_F(x))=0$ for each $x\in V(F)$. Let $x\in V(F)$. We label each vertex $x$ of $F$ with $b_u(x)$. We say a vertex is $0$-, $1$-, or $(-1)$-\emph{labeled} if it is labeled with $0$, $1$, or $-1$, respectively. We say a vertex is \emph{$\pm 1$-labeled} if it is labeled with either $-1$ or $1$. If $x$ has only $0$-labeled neighbors, then $b_u(N_F(x))=0$ holds trivially. Thus, we assume, without loss of generality that $x$ has some $\pm 1$-labeled neighbor $v$ and let $P$ be the unique path from $u$ to $v$ in $F$. By construction, $P$ is an $M$-alternating path of even length. Necessarily, $x$ is $M$-saturated, for otherwise the path $P$ followed by the edge $vx$ would be an $M$-augmenting path of $F$, contradicting the maximality of $M$ in $F$ (by virtue of Theorem~\ref{thm:Berge}). If $xv\notin M$, let $y$ be the vertex matched with $x$ in $M$; otherwise, $x$ is the vertex immediately preceding $v$ in $P$ and let $y$ be the vertex immediately preceding $x$ in $P$. By construction, one of $v$ and $y$ is $1$-labeled and the other is $(-1)$-labeled. If $x$ had a $\pm 1$-labeled neighbor $t$ different from both $v$ and $y$, then there would be two different paths from $u$ to $t$ in $F$ (the $M$-alternating path from $u$ to $t$ and the path formed by the path from $u$ to $x$ in $F$ followed by the edge $xt$, which is not $M$-alternating), contradicting that $F$ is a forest. Hence, the only $\pm 1$-labeled neighbors of $x$ are $v$ and $y$ and, consequently, $b_u(N_F(x))=1+(-1)=0$. This completes the proof of $b_u\in\N(F)$ for each $u\in U$.

If $u$ and $u'$ are different vertices in $U$, then $b_u(u')=0$, since otherwise $u'$ would be reachable from $u$ by an $M$-alternating path of even length, contradicting the fact that $u$ and $u'$ are $M$-unsaturated. $B$ is a set of linearly independent vectors of $\N(F)$ because for each $u\in U$, $b_u$ is the only vector of $B$ that is nonzero at $u$. Since, in addition, $\vert B\vert=\vert U\vert=\vert V(F)\vert-2\vert M\vert$ coincides with the dimension of $\N(F)$ (by Lemma~\ref{thm:nullity}), we conclude that $B$ is a null basis of $F$. Moreover, by construction, $B$ is a $\{-1,0,1\}$-null basis of $F$.

It only remains to show that all the nonzeros of the vectors in $B$ can be determined in $O(N)$ time. For that purpose, let $D$ be the digraph having the same vertices as $F$ and with an arc from $x$ to $y$ if and only if there is some $v\in V(F)$ such that $xv\in E(F)-M$ and $vy\in M$. Notice that it is possible to enumerate the arcs $(x,y)$ of $D$ by considering, for each $M$-saturated vertex $v$, the vertex $y$ such that $vy\in M$ and each $x\in N_F(v)-\{y\}$. Thus, $\vert E(D)\vert\leq\sum_{v\in V(F)}\vert N_F(v)\vert=2\vert E(F)\vert$. Hence, given $F$ and $M$, $D$ can be built in $O(n)$ time. Let $u\in U$. Clearly, $Q^+(u)$ (resp.\ $Q^-(u)$) is the set of vertices that can be reached from $u$ by directed paths in $D$ of even (resp.\ odd) length. Thus, for each $u\in U$, we may determine $Q^+(u)$ and $Q^-(u)$ by performing a breadth-first search from the vertex $u$ in $D$, which takes time proportional to the number of vertices reachable from $u$; i.e., $O(\vert Q^+(u)\vert+\vert Q^-(u)\vert)$ time. In fact, as $F$ is acylic, $D$ has no directed cycles and, consequently, for each $u\in U$, a breadth-first search from $u$ in $D$ reaches no vertex more than once. Hence, we can determine all the nonzeros of the vectors in $B$ in $O\left(\sum_{u\in U}(\vert Q^+(u)\vert+\vert Q^-(u)\vert)\right)=O(N)$ time, where $N$ is the total number of such nonzeros.\end{proof}

Reasoning in a similar way, we show that, given a forest $F$ on $n$ vertices, we can find in $O(n)$ time the set $S$ of vertices $x$ of $F$ for which there is a vector in the null space of $F$ that is nonzero at $x$.

\begin{cor}\label{cor:S} Let $F$ be a forest and $M$ be a maximum matching of $F$. Let $U$ be the set of $M$-unsaturated vertices of $U$, and, for each $u\in U$, $Q(u)$ be the set of vertices reachable from $u$ by $M$-alternating paths in $F$ of even length. If $S=\bigcup_{u\in U}Q(u)$, then $S$ is the set of vertices $x$ of $F$ such that $z(x)\neq 0$ for some $z\in\N(F)$. Moreover, given $F$ and $M$, $S$ can be found in $O(n)$ time, where $n=\vert V(F)\vert$.\end{cor}
\begin{proof} Let $Q^+(u)$ and $Q^-(u)$ for each $u\in U$ and $B=\{b_u\colon\,u\in U\}$ as in Theorem~\ref{thm:B}. Clearly, for each $x\in V(F)$, there is some $z\in\N(F)$ such that $z(x)\neq 0$ if and only if $b_u(x)\neq 0$ for some $u\in U$, which, by construction, is equivalent to $x\in\bigcup_{u\in U}(Q^+(u)\cup Q^-(u))=\bigcup_{u\in U}Q(u)=S$. It only remains to show that $S$ can be found in $O(n)$ time. For that purpose let us consider again the digraph $D$ having the same vertices as $F$ and with an arc from $x$ to $y$ if and only if there is some $v\in V(F)$ such that $xv\in E(F)-M$ and $vy\in M$. As discussed in the preceding proof, $O(\vert E(D)\vert)=O(n)$ and $D$ can be built from $F$ and $M$ in $O(n)$ time. Since, for each $u\in U$, $Q(u)$ is the set of vertices reachable from $u$ by directed paths in $D$, $S$ is the set of vertices reachable from $U$ by directed paths in $D$. Hence, $S$ can be found in $O(\vert V(D)\vert+\vert E(D)\vert)=O(n)$ time by performing a breadth-first search in $D$ with initial set of vertices $U$.\end{proof}

\section{Finding a $\{-1,0,1\}$- and sparsest null basis}\label{sec:sparsest}

In this section, we give our optimal-time algorithm for finding a $\{-1,0,1\}$- and sparsest null basis of any given forest. Throughout this section, we adopt the following notation:
\begin{itemize}\renewcommand\labelitemi{\normalfont\bfseries \textendash}
 \item $F$ is a forest and $n$ denotes the number of vertices of $F$.
 
 \item $S$ is the set of vertices $x$ of $F$ for which there some vector $z\in\N(F)$ that is nonzero at $x$.
 
 \item $G$ is the forest having no isolated vertices and whose edges are the edges of $F$ having at least one endpoint in $S$.
 
 \item $R$ is the set $V(G)-S$. 

 \item For any graph $H$, let $\NS(H)$ be the set of vectors in $\N(H)$ that are zero at every vertex of $H$ not in $S$.
 
 \item We regard each component of $G$ as rooted at some of its vertices in $S$. If $x$ is a vertex of $G$, we denote by $\kappa(x)$ the set of children of $x$ in $G$, by $\pi(x)$ the parent (if any) of $x$ in $G$, and by $\pi^2(x)$ the parent (if any) of $\pi(x)$ in $G$.

 \item Let $\beta:V(G)\to\mathbb N$ defined by:
 \[ \beta(x)=\begin{cases}
           \min\{\betadown(x),\beta(\pi(x))-\betadown(x)\}&\text{if }x\in R,\\
           \betadown(x)+\beta(\pi(x))&\text{if $x\in S$},
          \end{cases} \]
  where $\beta(\pi(x))$ is meant to be $0$ if $x$ has no parent and where $\betadown:V(G)\to\mathbb N$ is defined by:
  \[ \betadown(x)=\begin{cases}
               \min_{c\in\kappa(x)}\betadown(c)&\text{if }x\in R,\\
               1+\sum_{c\in\kappa(x)}\betadown(c)&\text{if }x\in S.
             \end{cases} \]
\end{itemize}

A few more definitions are given at the beginning of Subsections~\ref{ssec:w} and~\ref{ssec:beta-matchings}.

This section is organized as follows. In Subsection~\ref{ssec:SR}, we discuss some facts about the sets $S$ and $R$. In Subsection~\ref{ssec:w}, we show that $\beta$ serves as a lower bound on the number of nonzeros of some vectors that are nonzero at certain vertices. In Subsection~\ref{ssec:beta-matchings}, we introduce $\beta$-matchings and show how to compute, given $F$, a $\beta$-matching in $O(n)$ time. In subsection~\ref{ssec:main}, we show that certain values of $\beta$ count precisely the number of nonzeros in our proposed sparsest basis and we state and prove our main result.

\subsection{The sets $S$ and $R$}\label{ssec:SR}

The result below was first proved in~\cite{MR664692} (where vertices of $S$ were called $0$-essential; see the Example after Corollary~3.5 therein). A proof is given for completeness.

\begin{thm}[\cite{MR664692}]\label{lem:stable} $S$ is a stable set of $F$.\end{thm}
\begin{proof} Let $M$, $U$, and $Q(u)$ as in Corollary~\ref{cor:S}. Any edge with both endpoints in $S$ would have one endpoint in $Q(u)$ and another in $Q(u')$ for some $u,u'\in U$, implying either a cycle in $F$ (if $u=u'$) or that the path from $u$ to $u'$ in $F$ is $M$-augmenting (if $u\neq u'$), contradicting that $F$ is a forest or that $M$ is maximum.\end{proof}

Hence, $S$ is also a stable set of $G$ and, since each edge of $G$ has at least one endpoint in $S$, each edge of $G$ has one endpoint in $S$ and another in $R$. Notice that no vertex of $R$ is a leaf of $G$. In fact, if some $r\in R$ had only one neighbor $x$ in $G$ then, by construction, $N_F(r)\cap S=\{x\}$ and thus, for any $z\in\N(F)$ that is nonzero at $x$, $z(N_F(r))=z(N_F(r)\cap S)=z(x)\neq 0$, a contradiction.

The sets $S$ and $R$ were studied in connection with maximum matchings and maximum stable sets in~\cite{1708.00934} (where the vertices in $S$ were called supported vertices and the vertices in $R$ were called core vertices). The theorem below follows by combining Corollary~5.15 of~\cite{1708.00934} with Theorem~\ref{thm:nullity}. We include a derivation from Corollary~\ref{cor:S} for completeness.

\begin{thm}[\cite{1708.00934,MR0172271}]\label{thm:dT=S-R} $\dim\N(F)=\vert S\vert-\vert R\vert$.\end{thm}
\begin{proof} Let $M$, $U$, and $Q(u)$ as in Corollary~\ref{cor:S}. Each vertex of $r\in R$ is matched in $M$ (because $U\subseteq S$) to some vertex in $S$ (for if $rt\in M$ and $s\in(N_G(r)\cap S)-\{t\}$, then $s\in Q(u)$ for some $u\in U$ and, by definition, $t\in Q(u)\subseteq S$). Hence, as each $M$-saturated vertex in $S$ has a neighbor in $R$, $\dim\N(F)=\vert U\vert=\vert S\vert-\vert R\vert$.\end{proof}

The forest $G$ was first studied in connection with the structure of the matchings of a tree in \cite{1709.03865} (where the components of $G$ were called S-atoms).

\subsection{Bounding the number of nonzeros from below}\label{ssec:w}

We adopt the following definitions:
\begin{itemize}\renewcommand\labelitemi{\normalfont\bfseries \textendash}
  \item For each $x\in S$, we define the \emph{weight} $w(x)$ of $x$ as the minimum number of nonzeros among the vectors in $\NS(G)$ that are nonzero at $x$.

  \item For each $x\in V(G)$, let $\Gdown(x)$ be the subgraph of $G$ induced by $x$ and its descendants in $G$.
  
  \item For each $x\in V(G)$, we define the \emph{downwards weight} $w(x)$ of $x$ as follows:
  \begin{enumerate}[(i)]
    \item if $x\in R$, then $\wdown(x)=\min_{c\in\kappa(x)}\wdown(c)$.
 
    \item if $x\in S$, then $\wdown(x)$ is the minimum number of nonzeros among the vectors in $\NS(\Gdown(x))$ that are nonzero at $x$.
  \end{enumerate}
\end{itemize}

The main result of this subsection is Lemma~\ref{lem:w>=beta} showing that $\beta$ is a lower bound on $w$ for each $x\in S$.

\begin{rmk}\label{rmk:z} If $z$ is a vector in $\NS(G)$ that is nonzero at some $x\in S$, it is straightforward to verify that the restriction of $z$ to $V(\Gdown(x))$ belongs to $\NS(\Gdown(x))$. Hence, $z$ has at least $\wdown(x)$ nonzeros on $V(\Gdown(x))$ for every $x\in S$ such that $z(x)\neq 0$.\end{rmk}

In order to prove the main result of this subsection, we need the two preliminary below.

\begin{lem}\label{lem:wdown} For each $x\in V(G)$, $\wdown(x)=\betadown(x)$.\end{lem}
\begin{proof} We proceed by induction. Let $x\in V(G)$. If $x$ is a leaf, then $x\in S$ and, by definition, $\wdown(x)=1$ and also $\betadown(x)=1$. Suppose $x$ is not a leaf and, by induction, that the equality $\wdown(c)=\betadown(c)$ holds for each $c\in\kappa(x)$. If $x\in R$, then, by definition and induction hypothesis, $\wdown(x)=\min_{c\in\kappa(x)}\wdown(c)=\min_{c\in\kappa(x)}\betadown(c)=\betadown(x)$. Thus, we assume without loss of generality, that $x\in S$. Let $c_1,\ldots,c_k$ be the children of $x$ in $G$. For each $i\in\{1,\ldots,k\}$, let $d_i\in\kappa(c_i)$ such that $\wdown(d_i)=\wdown(c_i)$. Let $z_i\in\NS(\Gdown(d_i))$ such that $z_i(d_i)=-1$ and having precisely $\wdown(c_i)$ nonzeros. Hence, the function $z:V(G)\to\mathbb R$ such that $z(y)=1$ if $y=x$, $z(y)=z_i(y)$ if $x\in V(\Gdown(d_i))$, and $0$ otherwise, satisfies $z\in\NS(\Gdown(x))$ and $z(x)\neq 0$. This proves that $\wdown(z)\leq 1+\sum_{c\in\kappa(x)}\wdown(c)$. Conversely, let $z\in\NS(\Gdown(x))$ that is nonzero at $x$ and let $i\in\{1,\ldots,k\}$. As $z(x)\neq 0$ and $z(N_G(c_i))=0$, $z$ must be nonzero also at some $e_i\in\kappa(c_i)$. By Remark~\ref{rmk:z}, $z$ has at least $\wdown(e_i)$ zeros on $V(\Gdown(e_i))$. This proves that $\wdown(x)\geq 1+\sum_{i=1}^k\wdown(e_i)\geq 1+\sum_{c\in\kappa(x)}\wdown(c)$. We conclude that $\wdown(x)=1+\sum_{c\in\kappa(x)}\wdown(c)$. Therefore, by induction hypothesis and definition of $\beta$, $\wdown(x)=1+\sum_{c\in\kappa(x)}\betadown(c)=\betadown(x)$. This completes the proof of the lemma.\end{proof}

\begin{lem}\label{lem:w} For each $x\in S$ that has a parent in $G$,
\[ w(x)\geq\wdown(x)+\min\{\wdown(\pi(x)),w(\pi^2(x))-\wdown(\pi(x))\}. \]
\end{lem}
\begin{proof} Let $z\in\NS(G)$ that is nonzero at $x$ and having precisely $w(x)$ nonzeros. As $x\in S$, $\pi(x)\in R$ and thus $\pi(x)$ has some parent $\pi^2(x)$. We consider two cases.

\mbox{\textbf{Case 1:}~\emph{Suppose $z(\pi^2(x))=0$.}} Thus, since $z(x)\neq 0$ and $z(N_G(\pi(x)))=0$, there must be some $y\in\kappa(\pi(x))-\{x\}$ such that $z(y)\neq 0$. By Remark~\ref{rmk:z}, $z$ has at least $\wdown(x)$ (resp.\ $\wdown(y)$) nonzeros on $V(\Gdown(x))$ (resp.\ $V(\Gdown(y))$). Therefore, $w(x)\geq\wdown(x)+\wdown(y)\geq\wdown(x)+\wdown(\pi(x))$.

\mbox{\textbf{Case 2:}~\emph{Suppose $z(\pi^2(x))\neq 0$.}} We assume, without loss of generality, that $z(\pi^2(x))=1$.
By Remark~\ref{rmk:z}, $z$ has at least $\wdown(x)$ nonzeros on $V(\Gdown(x))$. Let $y\in\kappa(\pi(x))$ such that $\wdown(\pi(x))=\wdown(y)$ and let $z_y\in\NS(\Gdown(y))$ that is nonzero at $y$ and having precisely $\wdown(\pi(x))$ nonzeros. Without loss of generality, we assume that $z_y(y)=-1$. Hence, $z^*:V(G)\to\mathbb R$ defined by $z^*(t)=z_y(t)$ if $t\in V(\Gdown(y))$, $z^*(t)=0$ if $t\in V(\Gdown(\pi(x)))-V(\Gdown(y))$, and $z^*(t)=z(t)$ otherwise, is a vector in $\NS(G)$ being nonzero at $\pi^2(x)$ and having at most $w(x)-\wdown(x)+\wdown(\pi(x))$ nonzeros. Thus, by definition of $w$, $w(\pi^2(x))\leq w(x)-\wdown(x)+\wdown(\pi(x))$.

In both cases, $w(x)\geq\wdown(x)+\min\{\wdown(\pi(x)),w(\pi^2(x))-\wdown(\pi(x))\}$.\qed
\end{proof}

We now prove the main result of this subsection.

\begin{lem}\label{lem:w>=beta} For each $x\in S$, $w(x)\geq\beta(x)$.\end{lem}
\begin{proof} 
We proceed by induction. If $x$ has no parent in $G$, then $w(x)=\wdown(x)=\betadown(x)=\beta(x)$ by definition of $w$ and $\beta$ and Lemma~\ref{lem:wdown}. Thus, we assume, without loss of generality, that $x$ has some parent $\pi(x)$ and, by induction, that $w(\pi^2(x))\geq\beta(\pi^2(x))$. Hence, by Lemmas~\ref{lem:wdown} and~\ref{lem:w}, and because $\pi(x)\in R$,
\begin{align*}
 w(x)&\geq\wdown(x)+\min\{\wdown(\pi(x)),w(\pi^2(x))-\wdown(\pi(x))\}\\
     &\geq\betadown(x)+\min\{\betadown(\pi(x)),\beta(\pi^2(x))-\betadown(\pi(x))\}
      =\betadown(x)+\beta(\pi(x))=\beta(x).
\end{align*}
This completes the proof of the lemma.\end{proof}

\subsection{$\beta$-matchings}\label{ssec:beta-matchings}

In this subsection, we define $\beta$-matchings and show that a $\beta$-matching can be found in $O(n)$ time.

We adopt the following definitions:
\begin{itemize}\renewcommand\labelitemi{\normalfont\bfseries \textendash}
  \item If $x\in R$, a \emph{$\beta$-minimizer for $x$} is any vertex $y\in N_G(x)$ such that one of the following assertions holds:
    \begin{enumerate}[(i)]
      \item $y\in\kappa(x)$ and $\beta(x)=\betadown(y)$, or
      
      \item $y=\pi(x)$ and $\beta(x)=\beta(\pi(x))-\betadown(x)$.
  \end{enumerate}
  It follows by definition that each $x\in R$ has at least one $\beta$-minimizer.
  
  \item A \emph{$\beta$-matching} is any set $M=\{r\phi(r)\colon\,r\in R\}$ such that, for each $r\in R$, $\phi(r)$ is a $\beta$-minimizer for $r$.      
\end{itemize}

\begin{lem}\label{lem:beta-matching} Each $\beta$-matching is a maximum matching of $G$.\end{lem}
\begin{proof} Suppose, for a contradiction, that some $\beta$-matching is not a matching of $G$. Thus, some $s\in S$ is a $\beta$-minimizer for two different vertices $r_1$ and $r_2$ of $R$. As $s$ cannot have both $r_1$ and $r_2$ as parents, we assume, without loss of generality, that $s=\pi(r_2)$. As $s$ is a $\beta$-minimizer for $r_2$, $\beta(s)-\betadown(r_2)=\beta(r_2)\leq\betadown(r_2)$ and, in particular, $\beta(s)\leq 2\betadown(r_2)$.

If $s=\pi(r_1)$, then, by symmetry, $\beta(s)\leq 2\betadown(r_1)$ and, as a consequence, $1+\betadown(r_1)+\betadown(r_2)\leq\betadown(s)\leq\beta(s)\leq 2\min\{\betadown(r_1),\betadown(r_2)\}$, a contradiction. Hence, we assume, without loss of generality that $s\in\kappa(r_1)$. As $s$ is a $\beta$-minimizer for $r_1$, $\betadown(s)=\beta(r_1)=\betadown(r_1)$, which implies $\beta(s)=\betadown(s)+\beta(r_1)=2\betadown(s)$. Since, as shown in the preceding paragraph, $\beta(s)\leq 2\betadown(r_2)$, we conclude that $\betadown(r_2)\geq \betadown(s)$, which contradicts the definition of $\betadown(s)$. This contradiction proves that each $\beta$-matching is a matching of $G$.

As each edge of $G$ has one endpoint in $R$ and $M$ saturates all of $R$, $M$ is a maximum matching of $G$.\end{proof}

\begin{lem}\label{lem:beta-matching-O(n)} If $F$ is given, then the values of $\beta(x)$ for each $x\in V(G)$ and a $\beta$-matching can be found in $O(n)$ time.\end{lem}
\begin{proof} Suppose $F$ is given. By Corollary~\ref{cor:S}, $S$ can be found in $O(n)$ time. Then, also $G$ and $R$ can be easily found in $O(n)$ time. We can root the components of $G$ at some vertex of $S$ each, in $O(n)$ total time (e.g., by depth-first search). As the value of $\betadown$ at every leaf is $1$ and, for every non-leaf $x$, $\betadown(x)$ can be computed in $O(\vert\kappa(x)\vert)$ time from the values of $\betadown$ at its children, the values $\betadown(v)$ for every $v\in V(G)$ can be computed in $O(n)$ time by traversing $G$ in postorder. Moreover, since $\beta(x)=\betadown(x)$ if $x$ has no parent (by definition) and, for every vertex $x$ having a parent, the value $\beta(x)$ can be computed in $O(1)$ time from the values $\betadown(x)$ and $\beta(\pi(x))$, the values $\beta(x)$ for every $x\in V(G)$ can be computed in $O(n)$ time by first computing $\betadown$ in $O(n)$ time and then traversing $G$ in preorder. Therefore, a $\beta$-matching can be found easily in additional $O(n)$ time by finding a $\beta$-minimizer for each $r\in R$ (which can be accomplished by traversing the neighborhood of $r$ once).\end{proof}

\subsection{Main result}\label{ssec:main}

The lemma below shows that certain values of $\beta$ count the number of nonzeros of the vectors in the sparsest null basis we will propose in our main result (Theorem~\ref{thm:main}). Recall from Lemma~\ref{lem:beta-matching} that each $\beta$-matching is a maximum matching of $G$.

\begin{lem}\label{lem:q=beta} Let $M$ be a $\beta$-matching. If $u$ is an $M$-unsaturated vertex of $G$ and $Q(u)$ is the set of vertices reachable from $u$ in $G$ by $M$-alternating paths of even length, then $\vert Q(u)\vert=\beta(u)$. \end{lem}
\begin{proof} Let $M=\{r\phi(x)\colon\,r\in R\}$ be a $\beta$-matching. For each $x\in S$ (resp.\ $x\in R$), we denote by $Q(x)$ the set of vertices reachable from $x$ in $G$ by $M$-alternating paths of even (resp.\ odd) length starting with an edge not in $M$ (resp.\ an edge in $M$).

We say a vertex $x\in R$ is \emph{cascading} if $\phi(x)\in\kappa(x)$.
 
\textbf{Claim 1:} \emph{For each cascading vertex $x\in R$, $\vert Q(x)\vert=\beta(x)=\betadown(x)$.} We prove the claim by induction. Let $x\in R$ that is cascading. If $\phi(x)$ is a leaf, then $\vert Q(x)\vert=\vert\{\phi(x)\}\vert=1$ and also $\beta(x)=\betadown(x)=1$. Suppose that $\phi(x)$ is not a leaf and, by induction, that for all cascading vertices $y\in R$ that are descendants of $x$, $\vert Q(y)\vert=\betadown(y)$ holds. Since $M$ is a matching and $\phi(x)$ is matched with $x$ in $M$, each $c\in\kappa(\phi(x))$ is cascading and, by induction hypothesis, $\vert Q(c)\vert=\betadown(c)$. As $x\in R$ and $\phi(x)\in\kappa(x)$, $\beta(x)\leq\betadown(x)\leq\betadown(\phi(x))$. Moreover, as $\phi(x)$ is a $\beta$-minimizer for $x$ and $\phi(x)\neq\pi(x)$, necessarily $\beta(x)=\betadown(x)=\betadown(\phi(x))$.  Hence, by definition of $Q$ and $\betadown$:
\[ \vert Q(x)\vert=\vert Q(\phi(x))\vert
                      =1+\sum_{c\in\kappa(\phi(x))}\vert Q(c)\vert
                      =1+\sum_{c\in\kappa(\phi(x))}\betadown(c)
                      =\betadown(\phi(x))
                      =\beta(x)
                      =\betadown(x), \]
which completes the proof of Claim~1.

\textbf{Claim 2:} \emph{For each $x\in R$, $\vert Q(x)\vert=\beta(x)$.} Suppose, for a contradiction, that the claim is false and let $x$ be a vertex having the fewest ancestors among those vertices $y\in R$ such that $\vert Q(y)\vert\neq\beta(y)$. If $x$ is cascading, then, by Claim~1, $\vert Q(x)\vert=\beta(x)$. Hence, we assume, without loss of generality, that $x$ is not cascading; i.e., $\phi(x)=\pi(x)$. As $M$ is a matching, each $c\in\kappa(\pi(x))-\{x\}$ is cascading and, by Claim~1, $\vert Q(c)\vert=\betadown(c)$. By the choice of $x$, $\vert Q(\pi^2(x))\vert=\beta(\pi^2(x))$ (where if $\pi(x)$ has no parent, each of $\vert Q(\pi^2(x))\vert$ and $\beta(\pi^2(x))$ denotes $0$). Hence, by the definition of $Q$ and $\beta$ and because $\pi(x)$ is a $\beta$-minimizer for $x$:
\begin{align*}
  \vert Q(x)\vert&=\vert Q(\pi^2(x))\vert+1+\sum_{c\in\kappa(\pi(x))-\{x\}}\vert Q(c)\vert=\beta(\pi^2(x))+1+\sum_{c\in\kappa(\pi(x))-\{x\}}\betadown(c)\\
                 &=\beta(\pi^2(x))+\betadown(\pi(x))-\betadown(x)=\beta(\pi(x))-\betadown(x)=\beta(x),
\end{align*}
which contradicts the choice of $x$. This contradiction proves Claim~2.

Let $u$ be an $M$-unsaturated vertex of $G$. Thus, each $c\in\kappa(u)$ is cascading and, by Claim~1, $\vert Q(c)\vert=\betadown(c)$. Moreover, by Claim~2, $\vert Q(\pi(u))\vert=\beta(\pi(u))$ (where if $u$ has no parent, each of $\vert Q(\pi(u))\vert$ and $\beta(\pi(u))$ denotes $0$). Hence, by definition of $Q$, $\betadown$, and $\beta$,
\[ \vert Q(u)\vert=1+\sum_{c\in\kappa(u)}\vert Q(c)\vert+\vert Q(\pi(u))\vert
                  =1+\sum_{c\in\kappa(u)}\betadown(c)+\beta(\pi(u))=\betadown(u)+\beta(\pi(u))=\beta(u), \]
as desired.\end{proof}

We are now ready to prove our main result.

\begin{thm}\label{thm:main} Let $M$ be a $\beta$-matching. If $U$ is the set of $M$-unsaturated vertices of $G$ and, for each $u\in U$:
\begin{enumerate}[(i)]
 \item $Q^+(u)$ (resp.\ $Q^-(u)$) is the set of vertices reachable from $u$ by $M$-alternating paths in $G$ of length $4k$ (resp.\ $4k+2$) for any nonnegative integers $k$, and
 
 \item $b_u:V(F)\to\{-1,0,1\}$ such that $b_u(x)$ is $1$ if $x\in Q^+(u)$, $-1$ if $x\in Q^-(u)$, and $0$ otherwise,
\end{enumerate}
then $B=\{b_u\colon\,u\in U\}$ is a $\{-1,0,1\}$- and sparsest null basis of $F$ and, for each $u\in U$, $b_u$ is the only vector in $B$ that is nonzero at $u$ and $b_u$ is sparsest among the vertices in $\N(F)$ that are nonzero at $u$. Moreover, given $F$, $B$ can be found in $O(N)$ time, where $N$ is the number of nonzeros in any sparsest null basis of $F$.\end{thm}
\begin{proof} By Theorem~\ref{thm:B}, $B_G=\{b_u\vert_{V(G)}\colon\,u\in U\}$ (where $b_u\vert_{V(G)}$ denotes the restriction of $b_u$ to $V(G)$) is a $\{-1,0,1\}$-null basis of $\N(G)$ and, for each $u\in U$, $b_u$ is the only vector in $B$ that is nonzero at $u$. Notice that, by construction, $b_u\vert_{V(G)}\in\NS(G)$ and thus $B_G$ is a basis of $\NS(G)$. It is straightforward to verify that the mapping $f:\NS(G)\to\N(F)$ that assigns to each $z\in\NS(G)$ the vector $f(z)$ that arises from $z$ by extending it with the value $0$ at each vertex of $V(F)-V(G)$ is well-defined (i.e., that indeed $z'\in\N(F)$) and that $f$ is an injective linear map. Thus, as $b_u=f(b_u\vert_{V(G)})$ for each $u\in U$, $B$ is a set of linearly independent vectors of $\N(F)$. Moreover, as Theorem~\ref{thm:dT=S-R} ensures that $\dim\N(F)=\vert S\vert-\vert R\vert=\vert U\vert$, it turns out that $B$ is a basis of $\N(F)$ and the mapping $f$ is an isomorphism. Hence, the definition of $w$ implies that, for each $x\in S$, $w(x)$ is also the minimum number of nonzeros among the vectors in $\N(F)$ that are nonzero at $x$. Therefore, for each $u\in U$, $b_u$ is sparsest among the vectors in $\N(F)$ that are nonzero at $u$ because, by construction, the number of nonzeros of $b_u$ is $\vert Q^+(u)\vert+\vert Q^-(u)\vert=\beta(u)\leq w(u)$ (by Lemma~\ref{lem:q=beta} and Lemma~\ref{lem:w>=beta}, respectively).

We label the vertices of $U$ in such a way that $U=\{u_1,\ldots,u_d\}$ and the numbers of nonzeros in $b_{u_1},\ldots,b_{u_d}$ is nondecreasing. We claim that: \emph{for each $i\in\{1,\ldots,d\}$, $b_{u_i}$ is sparsest among the vectors in $\N(F)$ that are not in the subspace generated by $\{b_{u_1},\ldots,b_{u_{i-1}}\}$}. In order to prove the claim, let $i\in\{1,\ldots,d\}$ and let $z\in\N(F)$ such that $z$ is not in the subspace generated by $\{b_{u_1},\ldots,b_{u_{i-1}}\}$. As each $b_u$ is the only vector in the basis $B$ that is nonzero at $u$ and $z$ is not in the subspace generated by $\{b_{u_1},\ldots,b_{u_{i-1}}\}$, it follows that $z$ is nonzero at some vertex in $U$ but zero at each of $u_1,\ldots,u_{i-1}$. Thus, $z$ is nonzero at $u_j$ for some $j\geq i$. As $b_{u_j}$ is sparsest among the vectors in $\N(F)$ that are nonzero at $u_j$, $z$ has at least as many nonzeros as $b_j$, which in turn has at least as many nonzeros as $b_i$. This proves the claim and, by virtue of Theorem~\ref{thm:greedy}, $B$ is a sparsest null basis of $F$.

It only remains to prove the running time bound. Suppose $F$ is given. By Lemma~\ref{lem:beta-matching-O(n)}, a $\beta$-matching can be found in $O(n)$ time. Then, by Theorem~\ref{thm:B}, we can compute $B$ in additional $O(N)$ time, where $N$ is the number of nonzeros of $B$; i.e., the number of nonzeros in any sparsest null basis of $F$. As $N\geq n$, these two steps together take $O(N)$ total time.\end{proof}

\begin{cor} Given $F$, the number of nonzeros in any sparsest null basis of $F$ can be found in $O(n)$ time.\end{cor}
\begin{proof} Let $M$ be a $\beta$-matching and $B=\{b_u\colon\,u\in U\}$ be the sparsest null basis of $F$ considered in Theorem~\ref{thm:main}. As each $b_u$ has $\beta(u)$ nonzeros by Lemma~\ref{lem:q=beta}, the total number $N$ of nonzeros in $B$ is $\sum_{u\in U}\beta(u)$. As the values of $\beta$ for each $x\in V(G)$ and a $\beta$-matching can be computed in $O(n)$ time by Lemma~\ref{lem:beta-matching-O(n)}, $N$ can be found in $O(n)$ time.\end{proof}

\section*{Acknowledments}

This work was partially supported by the ``Red Argentino-Brasile\~na de Teor\'ia Algebraica y Algor\'itmica de Grafos, Etapa 2016'' (SPU-ME, Argentina) and SiDIU (UNSL). D.A.\ Jaume, G.\ Molina, and A.\ Pastine were partially supported by Universidad Nacional de San Luis, Grant PROIPRO 03-2216. M.D.~Safe was partially supported by ANPCyT PICT 2015-2218, CONICET PIO 14420140100027CO, and UNS PGI 24/ZL16.

\end{document}